\documentclass{article}
\usepackage{amsmath,amssymb,amsthm}
\usepackage[T1]{fontenc}
\usepackage[utf8]{inputenc}
\usepackage{lmodern}
\usepackage{microtype}
\usepackage{geometry}
\usepackage[backend=biber,bibencoding=utf8,maxbibnames=99]{biblatex}
\usepackage{varioref}
\usepackage{hyperref}
\usepackage{cleveref}
\usepackage{tikz-cd}

\theoremstyle{plain}
\newtheorem{thm}{Theorem}[section]
\newtheorem{prop}{Proposition}[section]
\newtheorem{lem}{Lemma}[section]

\renewcommand{\div}{\mathop{\mathrm{div}}}
\newcommand{\grad}{\mathop{\mathrm{grad}}}
\newcommand{\vol}{\mathop{\mathrm{vol}}}
\newcommand{\mesh}{\mathop{\mathrm{mesh}}}

\renewcommand{\vec}{\mathop{\mathrm{vec}}}
\newcommand{\sym}{\mathop{\mathrm{sym}}}
\newcommand{\skw}{\mathop{\mathrm{skw}}}

\newcommand{\curl}{\mathop{\mathrm{curl}}}
\newcommand{\half}{{\textstyle{\frac{1}{2}}}}

\newcommand{\tr}{\mathsf{T}}
\newcommand{\Tr}{\mathrm{Tr}}
\newcommand{\trace}{\mathop{\mathrm{tr}}}

\newcommand{\jump}[1]{[\![#1]\!]}

\newcommand{\descref}[1]{\hyperref[#1]{\rm #1}}

\title{%
  A mixed finite element for weakly-symmetric elasticity
}%

\author{Tobin Isaac\footnote{%
  \href{mailto:tisaac@cc.gatech.edu}{tisaac@cc.gatech.edu},
  School of Computational Science and Engineering, Georgia Institute of
  Technology.%
  }}

\begin{document}

\maketitle

\abstract{%
  We develop a finite element discretization for the weakly symmetric
  equations of linear elasticity on tetrahedral meshes.  The finite element
  combines, for $r \geq 0$, discontinuous polynomials of $r$ for the
  displacement, $H(\div)$-conforming polynomials of order $r+1$ for the
  stress, and $H(\curl)$-conforming polynomials of order $r+1$ for the vector
  representation of the multiplier.  We prove that this triplet is stable and
  has optimal approximation properties.  The lowest order case can be combined
  with inexact quadrature to eliminate the stress and multiplier variables,
  leaving a compact cell-centered finite volume scheme for the displacement.
}%

\section{Introduction}

\subsection{Symmetric linear elasticity}


We consider the equations of linear elasticity in mixed first-order form,
\begin{equation}\label{eq:elas}
  \begin{aligned}
    A\sigma &= \sym \grad u, &
    \div \sigma &= g, & & x\in\Omega,\\
    & & u &= 0, & & x\in\partial\Omega.
  \end{aligned}
\end{equation}
Here $\sigma$ is the stress tensor, $u$ is the displacement, $A$ is the
compliance tensor, and $\sym G := \half(G + G^\tr)$.  These equations have the
well-posed variational form:
\begin{equation}\label{eq:symwf}
  \begin{aligned}
    \int_\Omega (A \sigma : \tau + u \cdot \div \tau) + (\div \sigma - g)
    \cdot v\ dx &= 0,
    & & \tau \in H(\div;\mathbb{S}), &v \in L^2(\mathbb{V}).
  \end{aligned}
\end{equation}
Here $\mathbb{S}$ is the symmetric subspace of $\mathbb{M}$, the space of
$3\times 3$ matrices, $\mathbb{V}$ is $\mathbb{R}^3$, and $H(\div;\mathbb{S})$
and $L^2(\mathbb{V})$ have the usual meanings. We define the bilinear forms
\[
  \begin{aligned}
    a(\sigma,\tau) &:= \int_{\Omega} A\sigma : \tau\ dx, &
    b(\tau,v) &:= \int_{\Omega} \div \tau \cdot  v\ dx.
  \end{aligned}
\]
The well-posedness of \eqref{eq:symwf} is guaranteed by the coercivity of $a$ on
the kernel of $b$ and the fact that $b$ satisfies the
Ladyzhenskaya-Babu\v{s}ka-Brezzi (LBB) stability condition
\cite{Ladyzhenskaya1969,Babuska1973,Brezzi1974}.  We will use $\ker(f,X)$ to
refer to the kernel of operator $f$ on space $X$, and $\ker(g,X;Y)$ to refer
to the kernel of the bilinear form $g:X\times Y \to \mathbb{R}$ interpreted as
an operator in $\mathcal{L}(X;Y^*)$.  We thus state the well-posedness of
\eqref{eq:symwf} as the existence of constants $\alpha > 0$ and $\beta > 0$ such
that,
\[
  \begin{aligned}
    \inf_{\tau\in\ker(b,H(\div;\mathbb{S});L^2(\mathbb{V}))}
    a(\tau,\tau) &\geq \alpha \|\tau\|_{\div}^2, \\
    \inf_{v\in L^2(\mathbb{V})} \sup_{\tau \in H(\div;\mathbb{S})}
    b(\tau,v) &\geq \beta \|\tau\|_{\div}\|v\|_0.
  \end{aligned}
\]
We use $\|\cdot \|_{\{k,\div,\curl\},\Omega}$ for the $H^k(\Omega)$,
$H(\div;\Omega)$, and $H(\curl;\Omega)$ norms, dropping $\Omega$ where it is
clear from context, and $|\cdot|_{\{k,\div,\curl\},\Omega}$ for their
seminorms.

\Cref{eq:elas} conserves both linear momentum ($\div \sigma = g$) and
angular momentum ($\sigma = \sigma^\tr$).  Finite elements that conserve both
pointwise have recently been developed \cite{AdamsCockburn2004,
ArnoldAwanouWinther2008, ArnoldAwanouWinther2014, GopalakrishnanGuzman2011,
Hu2017, HuZhang2016}.  These discretizations are either high order
\cite{AdamsCockburn2004,ArnoldAwanouWinther2008,Hu2017}, nonconforming
\cite{ArnoldAwanouWinther2014,GopalakrishnanGuzman2011}, or require bubble
functions in the low-order case \cite{HuZhang2016}.  We set these methods
aside until our closing discussion.

\subsection{Weakly-symmetric linear elasticity}

We have more flexibility in discretization if we start from the
weakly-symmetric version of the variational equations, in which the symmetry
of $\sigma$ is enforced by a Lagrange multiplier $p$:
\begin{equation}\label{eq:hr}
  \begin{aligned}
    \int_\Omega (A \sigma : \tau + \div \tau \cdot u + \tau : p)\ dx &= 0,
    & & \tau \in H(\div;\mathbb{M}), \\
    \int_\Omega \div \sigma \cdot v\ dx &= \int_{\Omega} g \cdot v\ dx,
    & & v \in L^2(\mathbb{V}), \\
    \int_\Omega \sigma : q\ dx &= 0, & & q\in L^2(\mathbb{K}).
  \end{aligned}
\end{equation}
Here $\mathbb{K}$ is the skew-symmetric subspace of $\mathbb{M}$. Henceforth
we name the function spaces $\Sigma := H(\div;\mathbb{M})$, $V :=
L^2(\mathbb{V})$ and $Q := L^2(\mathbb{K})$. We assume that the original
compliance tensor $A:\mathbb{S}\to\mathbb{S}$ is extended continuously to
$A:\mathbb{M}\to\mathbb{M}$ so that $a$ is coercive on $\ker(\div,\Sigma)$. We
define the bilinear form $c(\tau,q) := \int_\Omega \skw \tau : q\ dx$, where
$\skw G := \half(G - G^\tr)$, which lets us rewrite \eqref{eq:hr} as
\[
  \begin{aligned}
  a(\sigma, \tau) + \{b(\tau,u) + c(\tau,p)\} + \{b(\sigma,v) + c(\sigma, q)\}
  &=
  (g,v), && (\tau, v, q) \in \Sigma \times V \times Q.
  \end{aligned}
\]
%

On the one hand, weakly-imposed symmetry is a simplification because $b$ is
now a vector-valued version of the familiar bilinear operator on
$H(\div;\mathbb{V})\times L^2(\mathbb{R})$ that originally motivated the
development of the LBB condition, and we have well-established choices of
discrete spaces that satisfy LBB stability for this operator.  On the other
hand, this is a complication because LBB stability must be satisfied with
respect to a larger constraint space, $V\times Q$.

\begin{prop}\label{prop:lbba}
  There exists $\gamma > 0$ such that
  %
  \[
    \inf_{(v,q)\in V \times Q}
    \sup_{\tau \in \Sigma}\ 
    b(\tau,v) + c(\tau,q) \geq \gamma
    \|\tau\|_{\div}(\|v\|_0 + \|q\|_0).
    \qed
  \]
  %
\end{prop}

%
It is simple to check that \cref{prop:lbba} is equivalent to the following
two proposition.
\begin{prop}
  \label{prop:lbbb}
  There exists $\tilde \gamma > 0$ such that
  %
  $
    \inf_{v\in V} \sup_{\tau \in \Sigma}\
    b(\tau,v) \geq \tilde \gamma \|\tau\|_{\div}\|v\|_0.
    \qed
  $
  %
\end{prop}

\begin{prop}
  \label{prop:lbbc}
  There exists $\hat{\gamma} > 0$ such that
  %
  $
    \inf_{q\in Q} \sup_{\tau \in \ker(b,\Sigma;V)}
    c(\tau,q)
    \geq \hat\gamma \|\tau\|_{\div}\|q\|_0.
  $
\end{prop}
This split is a useful form of \cref{prop:lbba} because, as mentioned above,
\cref{prop:lbbb} is a vector-valued generalization of a well established
result, and because we know $\curl [H(\curl;\mathbb{M})]\subseteq
\ker(b,\Sigma;V)$.

We provide a brief sketch of a proof of \cref{prop:lbbc}, because it illustrates
a method used to prove the stability of finite element approximations of
\eqref{eq:hr}.  The proof uses the operators $\vec:\mathbb{K}\to\mathbb{V}$
and $\Xi:\mathbb{M}\to\mathbb{M}$ defined by
\[
  \begin{aligned}
    (\vec A) \times b := Ab, & &
    \Xi A := A^\tr - \trace(A) I,
  \end{aligned}
\]
which have the following identities for sufficiently smooth fields:
\begin{align}
  \Xi^{-1}A &= A^\tr - \half\trace(A)I, \\
  \div \skw A  &= - \curl \vec \skw A \label{eq:divcurl}, \\
  \div \Xi \mu &= 2 \vec \skw \curl \mu. \label{eq:vecxi}
\end{align}
We state the following proposition about these two operators.

\begin{prop}\label{prop:xibiject}
  $\Xi$ induces a bounded linear bijection of $H^1(\mathbb{M})$ onto
  itself.\qed
\end{prop}

\begin{prop}\label{prop:vecxi}
  $\Xi$ induces a bounded linear map of $H(\curl;\mathbb{M})$ into
  $H(\div;\mathbb{M})$.\qed
\end{prop}

\begin{prop}\label{prop:divcurl}
  $\vec$ induces an isomorphism between $H(\div;\mathbb{K})$ and
  $H(\curl;\mathbb{V})$.\qed
\end{prop}

\begin{proof}[Proof of \cref{prop:lbbc}]\label{prf:lbbc}
  Let $q\in Q$ be given.  There exists  $\nu \in H^1(\mathbb{M})$ and a
  constant $C_1$ independent of $q$ such that $\div \nu = \vec q$ such that
  $\|\nu\|_1 \leq C_1 \|\vec q\|_0$
  \cite[Corollary~2.4]{GiraultRaviart1986}.  By \cref{prop:xibiject},
  $\mu := \Xi^{-1} \nu \in H^1(\mathbb{M}) \subset
  H(\curl;\mathbb{M})$ and thus $\curl \mu \in \ker(b,\Sigma;V)$.
  \Cref{prop:xibiject} also implies there exists $C_2$ such that
  \[
    \|\curl \mu \|_{\div} = \|\curl \mu\|_0 \leq C_2 \|\nu\|_1 \leq C_1 C_2
    \|\vec q\|_0.
  \]
  By \eqref{eq:vecxi},
  \[
    c(\curl \mu, q) =
    \int_{\Omega} \skw \curl \mu : q\ dx
    =
    \int_{\Omega} 2 \vec \skw \curl \mu \cdot \vec q\ dx =
    \int_{\Omega} \div \mu \cdot \vec q\ dx,
  \]
  and thus
  \[
    c(\curl \mu, q)  = \| \vec q\|_0^2 \geq \frac{1}{\sqrt{2}C_1 C_2} \|\curl
    \mu \|_{\div}\|q\|_0.
  \]
\end{proof}

The relationship between $\vec$ and $\Xi$ in \eqref{eq:vecxi} is fundamental
to the embedding of the elasticity complex into the de~Rham complex via the
Bernstein-Gelfand-Gelfand (BGG) resolution of the rigid-body motions.  We
refer the reader interested in more background on this result to the works of
\textcite{Eastwood2000, ArnoldFalkWinther2007}.
The latter authors use the BGG resolution to design stable
discretizations of \eqref{eq:hr} by choosing two finite element de~Rham
subcomplexes, one for $\mathbb{V}$ and one for $\mathbb{K}$, such that the
steps in the proof of \cref{prop:lbbc} above can be repeated at the discrete
level.  In particular, their method requires that there exists an approximate
operator $\Xi_h^{-1}$ mapping from the $H(\div;\mathbb{M})$ subspace in the
complex for $\mathbb{K}$ (the proxy for $\mathbb{K}$-valued 2-forms) into the
$H(\curl;\mathbb{M})$ subspace in the complex for $\mathbb{V}$ (the proxy
for $\mathbb{V}$-valued 1-forms) for which a discrete version of
\cref{prop:vecxi} is true.  Using this property, the authors prove the
stability of the following triplet of discrete subspaces $\Sigma_h$, $V_h$,
and $Q_h$ (alongside which we give the finite element exterior calculus
notation \cite{ArnoldFalkWinther2006} for the equivalent spaces of
differential forms):
%
\[
  \begin{aligned}
    \Sigma_h &:= \mathcal{P}_{r+1} (\mathcal{T}_h;\mathbb{M})
    \cap \Sigma &\sim& \mathcal{P}_{r+1} \Lambda^2
    (\mathcal{T}_h;\mathbb{V}),\\ V_h &:= \mathcal{P}_r
    (\mathcal{T}_h;\mathbb{V}) \cap V &\sim& \mathcal{P}_r \Lambda^3
    (\mathcal{T}_h;\mathbb{V}),\\ Q_h &:= \mathcal{P}_r
    (\mathcal{T}_h;\mathbb{K}) \cap Q &\sim& \mathcal{P}_r \Lambda^3
    (\mathcal{T}_h;\mathbb{K}). 
  \end{aligned}
\]
%
Here $\mathcal{T}_h$ is a tetrahedral mesh of the domain $\Omega$, i.e.\ a
simplicial complex.  We will denote by $\Delta_k(\mathcal{T}_h)$ the
$k$-dimensional simplices in $\mathcal{T}_h$, by $h_S$ the diameter of
simplex $S$, and $h := \max_{S \in \mathcal{T}_h} h_S$.

$\mathcal{P}_r(\mathcal{T}_h;\mathbb{X})$ is the space of functions that are
equal to $\mathbb{X}$-valued polynomials of degree at most $r$ when
restricted to each $T\in \Delta_3(\mathcal{T}_h)$, and $\mathcal{P}_r
\Lambda^k(\mathcal{T}_h;\mathbb{Y})$ are the same functions interpreted
instead as $\mathbb{Y}$-valued $k$-forms. 

The triplet $\Sigma_h \times V_h \times Q_h$ is by no means the only stable
discretization of \eqref{eq:hr}: see \cite[\S~1]{ArnoldFalkWinther2007} for a
survey.

\subsection{The proposed finite element}

In light of the $H(\div;\mathbb{K})\sim H(\curl;\mathbb{V})$ isomorphism, we
propose the triplet $\Sigma_h \times V_h \times \tilde Q_h$  for the
discretization of \eqref{eq:hr}, where
%
\[
  \begin{aligned}
    \tilde Q_h &:= \mathcal{P}_{r+1} (\mathcal{T}_h;\mathbb{K}) \cap
    H(\div;\mathbb{K})
    \sim \vec{}^{-1} [\mathcal{P}_{r+1} \Lambda^1 (\mathcal{T}_h;\mathbb{R})].
  \end{aligned}
\]
%
The stress and velocity spaces are the same as in the element of Arnold, Falk,
and Winther, but the multiplier space $\tilde Q_h$ is now
$H(\div;\mathbb{K})$-conforming and has the same order as $\Sigma_h$.  The
discrete version of \cref{prop:lbbb} is known to hold when restricted to
$\Sigma_h \times V_h$, so to prove the stability of our triplet in discretizing
\eqref{eq:hr} it only remains to prove a discrete version of
\cref{prop:lbbc}, which we do in \cref{sec:proof}.

In the lowest order case, our finite element can be efficiently reduced to a
generalized displacement method, in fact to a finite volume scheme with
pointwise second-order convergence at cell centers.  We discuss this aspect of
the finite element in \cref{sec:nodal}.

\section{Stability and convergence}
\label{sec:proof}

In this section we prove the stability and convergence properties of $\Sigma_h
\times V_h \times \tilde Q_h$ as a discrete setting for \eqref{eq:hr}.
\Cref{thm:lbbch} is the only component of the proof of stability that has not
already been established.

\begin{thm}[$c$ is stable on the kernel of $b$ for
  $\Sigma_h \times V_h \times \tilde{Q}_h$]\label{thm:lbbch}%
  Let $\mathcal{T}_h$ be a conformal tetrahedral mesh of $\Omega$ that is
  shape regular, i.e.\ there exists $C_{\text{mesh}}$ such that $h_T^3 \leq
  C_{\mathop{\mathrm{mesh}}}|\mathop{\mathrm{vol}} T|$, $T\in \mathcal{T}_h$.
  Then there exists
  $\gamma(\Omega,C_{\mathop{\mathrm{mesh}}}) > 0$ independent
  of $h$ such that
  \begin{equation}\label{eq:infsupch}
    \inf_{q_h\in \tilde Q_h} \sup_{\tau_h \in \ker(b,\Sigma_h;V_h)}
    c(\tau_h,q_h) \geq \gamma \|\tau_h\|_{\div}\|q_h\|_0.
  \end{equation}%
\end{thm}

The proof of \cref{thm:lbbch} is similar to the proof of
\cref{prop:lbbc} in that we will use only a subspace of $\ker(b,\Sigma_h;V_h)$
defined by the curl of another space,
\[ M_h := \mathcal{P}_{r + 2} (\mathcal{T}_h;\mathbb{M}) \cap
  H(\curl;\mathbb{M}) \sim \mathcal{P}_{r + 2} \Lambda^1
  (\mathcal{T}_h;\mathbb{V}).
\]
As $M_h \stackrel{\curl}{\to} \Sigma_h \stackrel{\div}{\to} V_h \to 0$ is
a subcomplex of the de~Rham complex, 
$\curl[M_h]\subseteq \ker(b,\Sigma_h;V_h)$.

The proof takes the form of the macroelement method of
\textcite{Stenberg1984,Stenberg1990}.  The macroelement method was developed
to prove the stability of $b$ on $(H_0^1(\mathbb{V}) \times
L_0^2(\mathbb{R}))$-conforming finite elements, but the structure of the method
requires few changes to apply to $M_h \times \tilde Q_h$.

First we define mesh dependent norms and seminorms:
\[
  \begin{aligned}
    \|q\|_{1,h}^2 &:=
    \sum_{T\in\Delta_3(\mathcal{T}_h)} h_T^2 |\vec q|_{1,T}^2 + 
    \sum_{f \in \Delta_2(\mathcal{T})} h_f \|\jump{\vec q}\|_{0,f}^2,\\
    \|\mu\|_{0,h}^2 &:= \sum_{T\in\Delta_3(\mathcal{T}_h)} h_T^{-2}
    \|\Xi \mu\|_{0,T}^2 + 
    \sum_{f \in \Delta_2(\Xi \mathcal{T}_h)} h_f^{-1} \|(\Xi \mu)n\|_{0,f}^2.
  \end{aligned}
\]
Here $\jump{\omega} := \omega^- - \omega^+$ is the jump of $\omega$ across $f =
T^-\cap T^+$ where $n$ points from $T^-$ to $T^+$ (and $\omega^+ := 0$ if
$f\subset \partial \Omega$).

%

\begin{prop}\label{lem:czeronorm}
  $|c(\mu,q_h)| \leq \|\mu\|_{0,h}\|q_h\|_{1,h}$, $\mu\in
  H(\curl;\mathbb{M})$, $q_h \in \tilde Q_h$.
\end{prop}

\begin{proof}
  Using \eqref{eq:divcurl} and integration by parts,
  \begin{equation}\label{eq:intbyparts}
    \begin{aligned}
      c(\curl \mu, q_h)
      &=
      c(\div \Xi \mu, \vec q_h)
      \\
      &= -\sum_{T\in \Delta_3(\mathcal{T}_h)} (\Xi \mu, \grad \vec q_h)_T
      +
      \sum_{f\in \Delta_2(\mathcal{T}_h)}
      \int_{f} \jump{\vec q_h}^\tr  (\Xi\mu) n\ ds,
    \end{aligned}
  \end{equation}
  so the bound is an application of the Cauchy-Schwarz inequality.
  %
\end{proof}

For each vertex $v\in \Delta_0(\mathcal{T}_h)$ we now define $\mathcal{T}_h^v$
to be the submesh of tetrahedra adjacent to $v$.  The shape regularity
constant $C_{\mesh}$ implies the existence of an upper bound on the number of
tetrahedra in $\mathcal{T}_h^v$, $|\Delta_3(\mathcal{T}_h^v)| \leq
k(C_{\mesh})$, so the simplicial complex topology of $\mathcal{T}_h^v$ is the
same as one of a finite dimensional set of reference macroelements,
$\{\mathcal{M}_h\}$.

We divide the facets of $\mathcal{T}_h^v$ by whether they are adjacent to $v$
or not,
\[
  \begin{aligned}
    \Gamma_i^v &:= \{f \in \Delta_2(\mathcal{T}_h^v):v \in \overline{f}\},
    &&
    \Gamma_o^v &:= \Delta_2(\mathcal{T}_h^v) \backslash \Gamma_i^v,
  \end{aligned}
\]
and we associate with $\mathcal{T}_h^v$ a subspace of $M_h$ that extends by zero
away from it,
\[
  \mathring{M}_h^v := \{ \mu_h \in M_h : \mu_h(x) = 0,\ x\in \Omega
    \backslash \overline{\mathcal{T}_h^v}; (\Xi\mu_h)n = 0, f\in \Gamma_o^v\}.
\]
$\mathring{M}_h^v$ is so defined that for any $q_h\in \tilde
Q_h$ the support of $c(\mu_h, q_h)$ is just the interior of $\mathcal{T}_h^v$,
\begin{equation}\label{eq:ctv}
  \begin{aligned}
    c(\curl \mu_h, q_h)
    &=
    -
    \sum_{T \in \Delta_3(\mathcal{T}_h^v)}
    \int_T \Xi \mu_h : \grad q_h\ dx
    +
    \sum_{f\in \Gamma_i^v}
    \int_{f} \jump{\vec q_h}^\tr (\Xi\mu) n\ ds, & &
    \mu_h\in\mathring{M}_h^v.
  \end{aligned}
\end{equation}

We restrict $\|\cdot \|_{1,h}$ to $\mathcal{T}_h^v$ using only the interior
facets,
\[
  |q|_{1,h,\mathcal{T}_h^v}^2 :=
  \sum_{T\in\Delta_3(\mathcal{T}_h^v)} h_T^2 |\vec q|_{1,T}^2 + 
  \sum_{f \in \Gamma_i^v} h_f \|\jump{\vec q}\|_{0,f}^2.
\]

\begin{prop}\label{prop:hnorm}
  If $|q_h|_{1,h,\mathcal{T}_h^v} = 0$, then $q_h$ is constant in
  $\mathcal{T}_h^v$.\qed
\end{prop}

\begin{prop}\label{prop:normequiv}
  There exists $C$ such that
      $\|q_h\|_{1,h}^2 \leq \sum_{v\in \Delta_0(\mathcal{T}_h)}
      |q_h|_{1,h,\mathcal{T}_h^v}^2 \leq C\|q_h\|_{1,h}^2$.\qed
\end{prop}


Nothing thus far has involved details of the spaces $M_h$ and $\tilde Q_h$.
We will use their properties for one key lemma.

\begin{lem}\label{lem:macroelement}
  If $q_h\in \ker(c,\tilde Q_h;\mathring{M}_h^v)$, then $q_h$ is constant in
  $\mathcal{T}_h^v$.
\end{lem}

\begin{proof}
  We will first show that if $q_h\in \ker(c,\tilde Q_h;\mathring{M}_h^v)$ then
  $q_h$ is constant in each tetrahedron in $\mathcal{T}_h^v$.  Let $e$ be an
  edge adjacent to $v$, let $v_e$ be the opposite vertex, and let $t_e$ be the
  tangent vector of $e$.  We claim that $\Xi^{-1}(\grad\vec q_ht_e t_e^\tr)$
  is a matrix field that is tangentially continuous across each facet $f$
  adjacent to $e$.  To verify this claim apply a tangent vector $t$ of $f$ to
  the field:
  \[
    \begin{aligned}
      \Xi^{-1}(\grad\vec q_ht_e t_e^\tr)t
      &= (t_e t_e^\tr (\grad \vec q_h)^\tr - \half \trace(\grad\vec q_h t_e t_e^\tr)
      I) t \\
      &= (t_e t_e^\tr (\grad \vec q_h)^\tr - \half (t_e^\tr \grad\vec q_h t_e)
      I) t \\
      &= (\partial_{t_e} (\vec q_h \cdot t))t_e - \half (\partial_{t_e} (\vec q_h
      \cdot t_e)) t.
    \end{aligned}
  \]
  Because $\vec q_h \cdot t_e$ and $\vec q_h \cdot t$ are continuous across $f$, and
  because the directional derivative $\partial_{t_e}$ is tangent to $f$, the
  resulting vector must be continuous across $f$.  

  We define a function $\mu_h^e$ supported on each tetrahedron $T$ adjacent to
  $e$ by
  \[
    \mu_h^e|_T :=
    \lambda_{v}\lambda_{v_e} \Xi^{-1} (\grad \vec q_h t_et_e^\tr),
  \]
  where $\lambda_w$ is the barycentric coordinate of vertex $w$ in $K$, which
  linearly interpolates between $1$ at $w$ and $0$ at the facet opposite $w$.
  Because $q_h$ is a polynomial of degree $r + 1$,
  $\mu_h^e$ is a polynomial of degree $r+2$; by design, $\mu_h^e$ has
  tangential continuity across the facets of $T$ adjacent to $e$ and vanishes
  on the others, so $\mu_h^e \in M_h$; because $e$ is adjacent to $v$, every
  facet $f$ adjacent to $e$ is in $\Gamma_i^e$, so $\mu_h^e$ vanishes on
  $\Gamma_o^v$, and thus $\mu_h^e \in \mathring{M}_h^v$.

  If $f\in \Gamma_i^v$ is adjacent to $e$, then by the fact that the scaling
  $\lambda_v \lambda_{v_e}$ must commute with the pointwise linear algebra,
  \[
    \jump{\vec q_h}^\tr (\Xi \lambda_v
    \lambda_{v_e} \Xi^{-1} (\grad q_h t_e t_e^\tr)) n = \lambda_v
    \lambda_{v_e} \jump{\vec q_h}^\tr \grad q_h t_e t_e^\tr n = 0.
  \]
  Therefore, by \eqref{eq:ctv},
  \[
    \begin{aligned}
      0 = c(\curl \mu_h^e, q_h)
      &=
      -
      \sum_{T\in \mathrm{adj}(e)}
      \int_T \Xi \lambda_v \lambda_{v_e} \Xi^{-1} \grad \vec q_h t_e t_e^\tr : \grad \vec
      q_h\ dx \\
      &=
      -
      \sum_{T\in \mathrm{adj}(e)}
      \int_T \lambda_v \lambda_{v_e} (\grad \vec q_h t_e t_e^\tr):(\grad \vec
      q_h)^\tr)\ dx \\
      &=
      -
      \sum_{T\in \mathrm{adj}(e)}
      \int_T \lambda_v \lambda_{v_e} |\partial_{t_e} \vec q_h|^2\ dx.
    \end{aligned}
  \]
  Because we can repeat this construction on every $e$, we must conclude that
  $\partial_{t_e} \vec q_h = 0$ on every tetrahedron
  $T\in\Delta_3(\mathcal{T}_h^v)$ and for every edge $e$ adjacent to $T$ and
  $v$.  But every $T$ must have three edges adjacent to $v$, whose tangents
  are linearly independent.  We conclude that $q_h$ is constant on every cell.

  If $q_h$ is constant on every cell, then $\jump{\vec q_h}\in
  \mathcal{P}_0(f;\mathbb{V})$ for each $f\in \Gamma_i^v$.  Let
  $\Upsilon:\Lambda^1(\mathbb{V})\to C^\infty(\mathbb{M})$ be
  the equivalence map between $\mathbb{V}$-valued 1-forms and matrices.  For
  each $f$ there must be some $\phi_f\in \mathcal{P}_0\Lambda^1(f;\mathbb{V})$
  such that
  \[
    \begin{aligned}
      \int_f \jump{\vec q_h}^\tr (\Xi \mu_h) n\ ds &=
      \int_f \phi_f \wedge \Tr_f (\Upsilon^{-1} (\mu_h)), && \mu_h \in
      M_h.
    \end{aligned}
  \]
  The degrees of freedom for $M_h$ must be equivalent to the canonical
  degrees of freedom for $\mathcal{P}_{r+2}\Lambda^1(T;\mathbb{V})$, which
  include all moments over facets of the form
  \[
    \begin{aligned}
      &\int_f \phi \wedge \Tr_f(\Upsilon^{-1}(\mu_h)), && \phi \in
      \mathcal{P}_{r+1}^{-1}\Lambda^1(f;\mathbb{V}).
    \end{aligned}
  \]
  For every $r\geq 0$, $\mathcal{P}_0 \Lambda^1(f;\mathbb{V})\subset
  \mathcal{P}_{r+1}^{-1} \Lambda^1(f;\mathbb{V})$.  Therefore for every facet
  there is a function $\mu_h^f\in M_h$ such that
  \[
    \int_f \jump{\vec q_h}^\tr (\Xi \mu_h^f) n\ ds =
    \int_f \phi_f \wedge \Tr_f (\Upsilon^{-1} (\mu_h^f)) = |\jump{\vec q_h}|,
  \]
  and such that $\Upsilon^{-1}(\mu_h^f)$ is in the kernel of all other
  functionals that define
  $\mathcal{P}_{r+2}\Lambda^1(\mathcal{T}_h;\mathbb{V})$.  This implies that
  $(\Xi\mu_h^f)n = 0$ on all facets $g\neq f$, and so $\mu_h^f \in
  \mathring{M}_h^v$.  By  \eqref{eq:ctv}, $0 = c(\curl \mu_h^v, q_h) =
  |\jump{\vec q_h}|$.  We conclude that $q_h$ is constant in
  $\mathcal{T}_h^v$.
\end{proof}

\begin{lem}\label{lem:beta}
  There exists $\beta(C_{\mesh}) > 0$ such that for each
  $v\in\Delta_0(\mathcal{T}_h^v)$
  \[
    \inf_{q_h\in \tilde Q_h} \sup_{\mu_h \in \mathring{M}_h^v}
    c(\curl \mu_h, q_h) \geq \beta |\mu_h|_{\curl}|q_h|_{1,h,v}.
  \]
\end{lem}

\begin{proof}
  (Cf. \textcite[Lemma~3.1]{Stenberg1984}.)

  Let $N_h^v\subset \tilde Q_h$ be the space of constant functions
  over $\mathcal{T}_h^v$.
  \Cref{lem:macroelement} proves that
  \begin{equation}\label{eq:cnorm}
    \inf_{q_h \in \tilde Q_h \backslash N_h^v}
    \sup_{\mu_h \in \mathring{M}_h^v}
    \frac{c(\curl \mu_h, q_h)}{|\mu_h|_{\curl}} > 0.
  \end{equation}
  Therefore \eqref{eq:cnorm} defines a norm on $\tilde Q_h \backslash N_h^v$.
  By \cref{prop:hnorm}, $|\cdot|_{1,h,\mathcal{T}_h^v}$ also defines a norm
  on $\tilde Q_h \backslash N_h^v$, so by the equivalence of finite
  dimensional norms there is $\beta_v > 0$ such that
  \[
    \inf_{q_h \in \tilde Q_h \backslash N_h^v}
    \sup_{\mu_h \in \mathring{M}_h^v}
    c(\curl \mu_h, q_h) > \beta_v
    |\mu_h|_{\curl}|q_h|_{1,h,\mathcal{T}_h^v}.
  \]
  All integrals in the definitions $c$, $|\cdot|_{\curl}$, and
  $|\cdot|_{1,h,\mathcal{T}_h^v}$ depend continuously on the Jacobians and
  inverse Jacobians of the mappings of the elements onto their corresponding
  elements in the reference macroelement $\mathcal{M}_h^v$, i.e.\
  $\beta_v(\{J_{v,T}\}\cup\{J_{v,T}^{-1}\})$ is a continuous function.  A
  scaling argument shows that $\beta_v$ is unchanged under a rescaling of
  coordinates, so we may assume $\det(J_{w,T_0}) = 1$ for every $w$.  Because
  of the mesh regularity constant $C_{\mesh}$, there is an upper bound
  $\sup_{w,T}\max\{|J_{w,T}|,|J^{-1}_{w,T}|\} \leq C(C_{\mesh})$.  Therefore
  there is a compact set containing $\{J_{w,T}\}\cup\{J_{w,T}^{-1}\}$ for
  all $w$ with the same macroelement as $v$, and so $\beta_w$ attains its
  minimum,
  \[
    \beta_{\mathcal{M}_h^v} := \min_{\mathcal{T}_h^w \sim
    \mathcal{M}_h^v}
    \beta_w > 0.
  \]
  Because there are finitely many macroelements, there must be a $\beta$
  such that $\beta_v \geq \beta > 0$ for all $v$.
\end{proof}

We are now ready to prove the main stability result.

\begin{proof}[Proof of \cref{thm:lbbch}]
  (Cf. \textcite[Theorem~2.1]{Stenberg1990}, \textcite[Proposition~3.3]{Verfurth1984}.)

  Let $q_h \in \tilde Q_h$ be given.  By \cref{lem:beta}, for each
  $v\in\mathcal{T}_h$ we can choose $\mu_h^v \in \mathring{M}_h^v$ such that
  $c(\curl \mu_h^v,q_h) \geq \half \beta |q_h|_{1,h,\mathcal{T}_h^v}^2$ 
  and $|\mu_h^v|_{\curl} \leq |q_h|_{1,h,\mathcal{T}_h^v}$.
  Defining $\mu_h := \sum_{v\in \Delta_0(\mathcal{T}_h)} \mu_h^v$, we have
  (using \cref{prop:normequiv})
  \[
    \begin{aligned}
      c(\curl \mu_h, q_h) &=
      \sum_{v\in \Delta_0(\mathcal{T}_h)} c(\curl \mu_h^v, q_h)
      \geq
      \half \beta \sum_{v\in \Delta_0(\mathcal{T}_h)}
      |q_h|_{1,h,\mathcal{T}_h^v}^2
      \geq \half \beta \|q_h\|_{1,h}^2,\\
      |\mu_h|_{\curl}^2 &\leq
      \sum_{v\in \Delta_0(\mathcal{T}_h)}
      |\mu_h^v|_{\curl}^2
      \leq
      \sum_{v\in \Delta_0(\mathcal{T}_h)}
      |q_h|_{1,h,\mathcal{T}_h^v}^2
      \leq
      C \|q_h\|_{1,h}^2.
    \end{aligned}
  \]
  Therefore there is a constant $C_1(C_{\mesh}) > 0$ such that
  \[
    c(\curl \mu_h, q_h) \geq
    C_1 |\mu_h|_{\curl} \|q_h\|_{1,h}
    =
    C_1 \frac{\|q_h\|_{1,h}}{\|q_h\|_0} |\mu_h|_{\curl} \|q_h\|_0.
  \]

  As a corollary to the proof of  \cref{prop:lbbc}, there exists
  $\mu^* \in H^1(\mathbb{M})$ and $C_2(\Omega)$ such that
  $c(\curl \mu^*,q_h) \geq C_2 \|q_h\|_0^2$ and
  $\|\mu^*\|_1\leq \|q_h\|_0$.  Let $\Pi_h^1 \mu^*$ be the smoothed
  projection \cite[\S~5.4]{ArnoldFalkWinther2006} of $H^1(\mathbb{M})$ into
  $\mathcal{P}_{r+2}(\mathcal{T}_h;\mathbb{M})\cap
  H^1(\mathbb{M})\subset M_h$, which satisfies
  \begin{align}\label{eq:smoothed}
    \|\Pi_h^1 \mu^* - \mu^*\|_{0,T} \leq C(C_{\mesh}) h_T \|\mu^*\|_{1,T}, & &
    |\Pi_h^1 \mu^*|_{1,T} \leq C(C_{\mesh}) \|\mu^*\|_{1,T}.
  \end{align}
  We have by \cref{lem:czeronorm}
  \begin{equation}\label{eq:ctrue}
    \begin{aligned}
      c(\curl \Pi_h^1 \mu^*, q_h) &= c(\curl \mu^*,q_h) + c(\curl (\Pi_h^1 \mu^* -
      \mu^*),q_h) \\
      &\geq C_2 \|q_h\|_0^2 - C \|\Pi_h^1 \mu^* - \mu^*\|_{0,h}
      \|q_h\|_{1,h}.
    \end{aligned}
  \end{equation}
  By the standard trace inequality for $H^1(T)$, $\|(\Xi\mu)n\|_{0,\partial
  T}^2 \leq C(C_{\mesh})(h_T^{-1}\|\mu\|_{0,T}^2 + h_T|\mu|_{1,T}^2)$, and
  so by \eqref{eq:smoothed},
  \[
    \begin{aligned}
      \|\Pi_h^1 \mu^* - \mu^*\|_{0,h}^2 =
      \sum_{T \in \Delta_3(\mathcal{T}_h)}
      h_T^{-2} \|\Pi_h^1 \mu^* - \mu^*\|_{0,T}^2 +
      \sum_{f \in \Delta_2(\mathcal{T}_h)}
      h_f^{-1} \|\Xi(\Pi_h^1 \mu^* - \mu^*)n\|_{0,f}^2& \\
      \leq
      C
      \sum_{T \in \Delta_3(\mathcal{T}_h)}
      h_T^{-2}\|\Pi_h^1 \mu^* - \mu^*\|_{0,T}^2 + |\Pi_h^1 \mu^* - \mu^*|_{1,T}^2&
      \leq
      C \|\mu^*\|_1^2.
    \end{aligned}
  \]
  Therefore there exist $C_3(\Omega,C_{\mesh})$ and $C_4(C_{\mesh})$ such that
  \[
    c(\curl \Pi_h^1 \mu^*, q_h)
    \geq
    \|q_h\|_0 (C_2 \|\mu^*\|_1 - C \|\mu^*\|_1 \frac{\|q_h\|_{1,h}}{\|q_h\|_0})
    \geq
    \|q_h\|_0 |\Pi_h^1 \mu^*|_{\curl} (C_3  - C_4 \frac{\|q_h\|_{1,h}}{\|q_h\|_0}).
  \]
  Letting $t = \|q_h\|_{1,h}/\|q_h\|_0$ and combining these two bounds, we get
  \[
    \inf_{q_h \in \tilde Q_h}
    \sup_{\mu \in M_h}
    c(\curl \mu_h, q_h)
    \geq
     |\mu_h|_{\curl}\|q_h\|_0
    \min_{t > 0} \max \{C_1 t, C_3 - C_4 t\} =
     |\mu_h|_{\curl}\|q_h\|_0
     \frac{C_1 C_3}{C_1 + C_4}.
  \]
\end{proof}

Having established the discrete stability of our conforming mixed finite
element approximation, the standard convergence estimates apply.

\begin{thm}\label{thm:convergence}
  (Cf. \textcite[Theorem~7.2]{ArnoldFalkWinther2007}.)
  Suppose $(\sigma, u, p)$ is the solution of \eqref{eq:hr} and $(\sigma_h,
  u_h, p_h)$ is the Galerkin solution in $\Sigma_h \times V_h \times \tilde
  Q_h$.  Let $P_h$ be the $L^2$ projection onto $V_h$.
  There exists $C$ independent of $h$ such that
  \[
    \begin{aligned}
      \|\sigma - \sigma_h \|_{\div} + \|u - u_h\|_0 + \|p - p_h\|_0
      &\leq C \inf_{\stackrel{(\tau_h, v_h, q_h)}{\in \Sigma_h\times V_h
      \times \tilde Q_h}}
      \|\sigma - \tau_h\|_{\div} + \|u - v_h\|_0 + \|p - q_h\|_0, \\
      \|\sigma - \sigma_h\|_0 + \|p - p_h\|_0 + \|u_h - P_{V_h} u\|_0 &\leq 
      C
      \inf_{\stackrel{(\tau_h, q_h)}{\in \Sigma_h\times \tilde Q_h}}
      \|\sigma - \tau_h \|_0 + \|p - q_h\|_0, \\
      |\sigma - \sigma_h|_{\div} &= \|(I - P_{V_h})g\|_0.
    \end{aligned}
  \]
  If $u$ and $\sigma$ are sufficiently smooth, then
  \[
    \begin{aligned}
      \|\sigma - \sigma_h\|_0 + \|p - p_h\|_0 + \|u_h - P_{V_h}u_h\|_0 &\leq
      C h^{r+2} \|u\|_{r + 3},\\
      \|u - u_h\|_0 &\leq C h^{r + 1}\|u\|_{r + 2},\\
      |\sigma - \sigma_h|_{\div} &\leq C h^{r + 1}\|\div
      \sigma\|_{r + 1}.
    \end{aligned}
  \]
\end{thm}

Because $\tilde Q_h$ has the same approximation order as $\Sigma_h$, an
improved estimate is achieved compared to the $\Sigma_h \times V_h \times Q_h$
element.

\section{Reduction to finite volume method}
\label{sec:nodal}

Let $\mathsf{A}$, $\mathsf{B}$, and $\mathsf{C}$ be the assembled matrices for
the bilinear forms $a$, $b$, and $c$ in $\Sigma_h \times V_h \times \tilde
Q_h$, let $\mathsf{g}$ be the assembled right-hand side, and let
$[\mathsf{s};\mathsf{u};\mathsf{p}]$ be the corresponding vector of
degrees of freedom for $(\sigma_h, u_h, p_h)$.  The matrix equation solved by
$(\sigma_h, u_h, p_h)$ is
\begin{equation}\label{eq:hrmat}
  \begin{bmatrix}
    \mathsf{A} & \mathsf{B}^\tr & \mathsf{C}^\tr \\
    \mathsf{B} & 0 & 0 \\
    \mathsf{C} & 0 & 0
  \end{bmatrix}
  \begin{bmatrix}
    \mathsf{s} \\ \mathsf{u} \\ \mathsf{p}
  \end{bmatrix}
  =
  \begin{bmatrix}
    0 \\ \mathsf{g} \\ 0
  \end{bmatrix}.
\end{equation}
Defining $\mathsf{Z} := \mathsf{A}^{-1} + \mathsf{A}^{-1}\mathsf{C}^\tr
(-\mathsf{C}\mathsf{A}^{-1}\mathsf{C}^\tr)^{-1}\mathsf{C}\mathsf{A}^{-1}$,
$\mathsf{u}$ solves $-\mathsf{B} \mathsf{Z}^{-1} \mathsf{B}^\tr \mathsf{u} =
\mathsf{g}$.  It is not practical to form this double Schur complement for
general discretizations.  For our lowest order $\Sigma_h \times \tilde Q_h$
space, however, a sparse approximation to $\mathsf{Z}^{-1}$ can be constructed
with nodal degrees of freedom and quadrature.  For the remainder of this
section $r=0$. 

(Our approach is inspired by the Multipoint Flux Mixed Finite
Element (MFMFE) method of \textcite{WheelerIotov2006}, and can be seen as a
vector-valued extension of it.)

Because $\Sigma_h \times \tilde Q_h$ restricts to
$\mathcal{P}_{1}(T;\mathbb{M} \times \mathbb{V})$ for each tetrahedron $T$,
we can choose a nodal basis for $\Sigma_h \times \tilde Q_h$: instead of
the moments on facets and edges, we take values at the functions at the
boundary vertices of those facets and edges.

We define a corner quadrature rule for approximating the integrals in $a$ and
$c$:
\[ \begin{aligned} a_h (\sigma_h, \tau_h) &:= \sum_{T\in
    \Delta_3(\mathcal{T}_h)} \frac{|\mathop{\mathrm{vol}} T|}{4} \sum_{v\in
    \Delta_0(\mathcal{T}_h) \cap \overline{T}} \left\{\left(A|_{T}(v)
        \tau_h|_T (v)\right) : \sigma_h|_T (v)\right\}, \\ c_h (\tau_h, q_h)
    &:= \sum_{T\in \Delta_3(\mathcal{T}_h)} \frac{|\mathop{\mathrm{vol}}
    T|}{4} \sum_{v\in \Delta_0(\mathcal{T}_h) \cap \overline{T}} \tau_h|_T (v)
    : q_h|_T (v).  \end{aligned} \]
%
The approximate matrices assembled from $a_h$ and $c_h$ with our nodal degrees
of freedom, $\mathsf{\check{A}}$ and $\mathsf{\check{C}}$, are block diagonal,
with one block for each vertex $v$ in $\mathcal{T}_h$.  The blocks will not
have uniform size: the size of each block will depend on how many facets and
edges are adjacent to $v$.  This block structure extends to the approximation
$\mathsf{\check Z}$ of $\mathsf{Z}$, so that computing $\mathsf{\check
Z}^{-1}$ involves only the solution of small, independent saddle-point systems
at each vertex, and $-\mathsf{B}\mathsf{\check Z}^{-1} \mathsf{B}^\tr$ is
sparse, connecting cell-centered $\mathsf{u}$ degrees of freedom to
vertex-adjacent neighbors.

%
%
%
%
%

Because corner quadrature exactly integrates linear polynomials, we can
combine the Bramble-Hilbert lemma with bounded projections to bound the loss
of accuracy due to quadrature error.

\begin{prop}\label{prop:quaderror}
  Given $\xi \in \Sigma \times Q$, let $\xi_\Sigma$ and $\xi_Q$ be its
  components, and let $\Pi_h \xi := (\Pi_h^{\div} \xi_\Sigma, \tilde P_h
  \xi_Q)$, where $\Pi_h^{\div}$ is the smoothed 2-form projection onto
  $\Sigma_h$ from \cite{ArnoldFalkWinther2006} and $\tilde P_h$ is the $L^2$
  projection on $\tilde Q_h$.  Let $f(\xi,\zeta) := a(\xi_\Sigma,\zeta_\Sigma)
  + c(\xi_\Sigma,\zeta_Q) + c(\zeta_\Sigma,\xi_Q)$ and $f_h(\xi_h,\zeta_h) :=
  a(\xi_{h,\Sigma},\zeta_{h,\Sigma}) + c(\xi_{h,\Sigma},\zeta_{h,Q}) +
  c(\zeta_{h,\Sigma},\xi_{h,Q})$.  Let $E(A,\xi_h, \zeta_h) :=
  f(\xi_h,\zeta_h) - f_h(\xi_h,\zeta_h)$, where $A$ is the compliance field,
  and let $E_c(\tau_h,q_h) := c(\tau_h,q_h) - c_h(\tau_h, q_h)$.
  There exists $C$ independent of $h$ such that
  \begin{align}
    |E(A,\Pi_h \xi,\zeta_h)| &\leq C h\|A\|_{1,\infty}
    \|\xi\|_1 \|\zeta_h\|_0, 
    & &
    \xi \in H^1(\mathbb{M}\times\mathbb{V}),
    \zeta_h \in \Sigma_h \times \tilde Q_h,\label{eq:quaderrorh}\\
    |E(A,\Pi_h\xi,\Pi_h \zeta)| &\leq C h^2\|A\|_{2,\infty}
    \|\xi\|_1 \|\zeta\|_1, 
    & &
    \xi,\zeta \in H^1(\mathbb{M}\times\mathbb{V}),\label{eq:quaderrorh2}\\
    |E_c(\tau_h,q_h)| &\leq C \|\tau_h\|_0 \|q_h\|_{1,h}, 
    & &
    \tau \in H^1(\mathbb{M}), q_h \in \tilde Q_h.\label{eq:quaderrorc}
  \end{align}
  Here $\|\cdot\|_{k,\infty}$ is the norm in $W^{k,\infty}(\Omega)$.
\end{prop}

\begin{proof}%
  (Cf. \textcite[Lemma~3.5, Lemma~4.2]{WheelerIotov2006}.)

  Let $P_0$ be the $L^2$ projection onto
  $\mathcal{P}_0(\mathcal{T}_h;\mathbb{X})$ for each field $\mathbb{X}$.
  We have
  \[
    E(A,\Pi_h\xi,\zeta_h) = 
    E((I-P_0)A,\Pi_h\xi,\zeta_h) +
    E(P_0 A,(I - P_0)\Pi_h\xi,\zeta_h) +
    E(P_0 A,P_0\Pi_h \xi,\zeta_h).
  \]
  The last error term is zero because the product is linear.
  Therefore
  \[
    E(A,\xi,\zeta_h)
    \leq
    C \|\zeta_h\|_0 (\|(I-P_0)A\|_{0,\infty}\|\xi\|_0
    +\|P_0 A\|_{0,\infty}\|(I - P_0)\Pi_h\xi\|_0).
  \]
  We have $\|(I-P_0)A\|_{0,\infty}\leq C h \|A\|_{1,\infty}$, and
  \[
    \|(I - P_0)\Pi_h \xi\|_0 \leq
    \|(I - \Pi_h) \xi\|_0 +
    \|(I - P_0) \xi\|_0 +
    \|P_0(I - \Pi_h) \xi\|_0 \leq C h \|\xi\|_1,
  \]
  which proves \eqref{eq:quaderrorh}.

  To prove \eqref{eq:quaderrorh2}, we expand the error as
  \[
    \begin{aligned}
      E(A,\Pi_h\xi,\Pi_h \zeta)
      =\
      &E(A,P_0\Pi_h\xi,P_0 \Pi_h \zeta) +
      \\
      &E((I-P_0)A,(I-P_0)\Pi_h\xi,\Pi_h \zeta) +
      E((I-P_0)A,P_0\Pi_h\xi,(I-P_0)\Pi_h \zeta) +
      \\
      &E(P_0 A,(I - P_0)\Pi_h\xi,(I - P_0)\Pi_h \zeta) +
      \\
      &E(P_0 A,P_0 \Pi_h\xi,(I - P_0)\Pi_h \zeta) +
      E(P_0 A,(I-P_0)\Pi_h\xi,P_0 \Pi_h \zeta).
    \end{aligned}
  \]
  The first term is exact for $A\in \mathcal{P}_1(\mathcal{T}_h)$, and so is
  $o(h^2 \|A\|_{2,\infty})$; the last two error terms are zero because their
  products are linear.
  Therefore
  \[
    \begin{aligned}
      E(A,\Pi_h\xi,\Pi_h \zeta)
      \leq
      C h^2(&\|A\|_{2,\infty}\|\Pi_h \xi\|_0 \|\Pi_h\zeta\|_0 +
      \\
      &\|A\|_{1,\infty}(\|\xi\|_1\|\Pi_h\zeta\|_0 + \|\Pi_h\xi\|_0\|\zeta\|_1) +
      \\
      &\|A\|_0\|\xi\|_1 \|\zeta\|_1).
    \end{aligned}
  \]

  The last bound \eqref{eq:quaderrorc} follows from $E_c(\tau_h, q_h) =
  E_c(\tau_h,(I-P_0)q_h)$.
\end{proof}

Unlike the MFMFE method, inexact quadrature could potentially affect the
stability of the discretization.  We must prove that this is not the case.

\begin{prop}\label{prop:lbbchquad}
  \Cref{thm:lbbch} holds for $c_h$ (though with a different bound
  $\gamma$).
\end{prop}
\begin{proof}
  Given $q_h$, let $\mu^*$ be as in the proof of \cref{thm:lbbch}. 
  We can modify \eqref{eq:ctrue} for $c_h$ using the bound
  \eqref{eq:quaderrorc}:
  \[
    \begin{aligned}
      c_h(\Pi_h^1 \mu^*,q_h) &=
      c(\Pi_h^1 \mu^*,q_h) - E_c(\Pi_h^1 \mu^*,q_h) \\
      &\geq \|q_h\|_0 |\Pi_h^1 \mu^*|_{\curl} (C_3  - C_4
      \frac{\|q_h\|_{1,h}}{\|q_h\|_0}) - |E_c(\Pi_h^1 \mu^*,q_h)|
      \\
      &\geq \|q_h\|_0 |\Pi_h^1 \mu^*|_{\curl}(C_3 - (C_4 + C)
      \frac{\|q_h\|_{1,h}}{\|q_h\|_0}).
    \end{aligned}
  \]
  Therefore to complete the proof using the macroelement technique, all that
  needs to be done is to verify that \cref{lem:macroelement} holds for $c_h$:
  that if $q_h \in \ker(c_h, \tilde Q_h;\mathring{M}_h^v)$, then $q_h$ is
  constant in $\mathcal{T}_h^v$. 

  Suppose $q_h\in \ker(c_h, \tilde Q_h;\mathring{M}_h^v)$.
  Now as before we define
  $\mu_h^e$ by
  \[
    \mu_h^e|_T :=
    \lambda_{v}\lambda_{v_e} \Xi^{-1} (\grad \vec q_h t_et_e^\tr),
  \]
  and observe
  \[
    \begin{aligned}
      \div (\Xi\mu_h^e) &=
      \div (\lambda_v \lambda_{v_e} (\grad \vec q_h t_e t_e^\tr)) \\
      &=\lambda_v \lambda_{v_e} \div (\grad \vec q_h t_e t_e^\tr) +
      (\grad \vec q_h t_e t_e^\tr) (\lambda_v \grad \lambda_{v_e} +
      \lambda_{v_e} \grad \lambda_v) \\
      &=\lambda_v \lambda_{v_e} \div (\grad \vec q_h t_e t_e^\tr) +
      (\lambda_v \partial_{t_e} \lambda_{v_e} +
      \lambda_{v_e} \partial_{t_e} \lambda_v)\partial_{t_e} \vec q_h.
    \end{aligned}
  \]
  The function $\lambda_v \lambda_{v_e}$ is zero at each corner of $T$,
  so that term is ignored in the computation of $c_h$. Combining this with the
  fact that the $\lambda_{v_1}(v_2) = \delta_{v_1,v_2}$,
  we get the following simplification for the
  contribution of $T$ to $c_h(\mu_h^e,q_h)$:
  \[
    \begin{aligned}
      - \frac{|\vol T|}{4}
      &\sum_{w\in \Delta_0(\mathcal{T}_h^v) \cap \overline{T}}
      (\lambda_v \partial_{t_e} \lambda_{v_e} +
      \lambda_{v_e} \partial_{t_e} \lambda_{v})\partial_{t_e} \vec q_h \cdot
      \vec q_h (w) \\=\
      &-\frac{|\vol T|}{4} \partial_{t_e} \vec q_h \cdot (\partial_{t_e} \lambda_{v_e} \vec q_h (v) + 
      \partial_{t_e} \lambda_v \vec q_h (v_e)).
    \end{aligned}
  \]
  Because $q_h$ is linear, we expand $q_h(x)$ around the
  midpoint of $e$, $(v + v_e)/2$:
  \[
    \begin{aligned}
      \vec q_h (x)
      &=
      \vec q_h((v+v_e)/2) + \grad \vec q_h \cdot (x - (v + v_e)/2),
    \end{aligned}
  \]
  by which
  \[
    \begin{aligned}
      &\partial_{t_e} \vec q_h \cdot (\partial_{t_e} \lambda_{v_e} \vec q_h (v) + 
      \partial_{t_e} \lambda_v \vec q_h (v_e))
      \\
      =\
      &\partial_{t_e} \vec q_h \cdot(\partial_{t_e}(\lambda_v + \lambda_{v_e})\vec q_h((v + v_e)/2) +
      \frac{h_e}{2}\partial_{t_e}(\lambda_v - \lambda_{v_e}) \partial_{t_e} \vec
      q_h).
    \end{aligned}
  \]
  Now $\lambda_v - \lambda_{v_e}$ is linear and is equal to $1$ at $v$ and
  $-1$ at $v_e$, so $\partial_{t_e} (\lambda_v - \lambda_{v_e}) = 2 / h_e$;
  likewise, $\partial_{t_e}(\lambda_v + \lambda_{v_e}) = 0$.
  Therefore
  \[
    \begin{aligned}
      \partial_{t_e} \vec q_h \cdot(\vec q_h((v + v_e)/2) +
      \frac{h_e}{2}(\partial_{t_e}(\lambda_v - \lambda_{v_e}) \partial_{t_e} \vec
      q_h))
      = 
      &|\partial_{t_e} \vec q_h|^2.
    \end{aligned}
  \]
  As before, we conclude that $q_h$ must be constant on each tetrahedron.  But
  $c_h(\mu_h,q_h)=c(\mu_h,q_h)$ in this case, so $q_h$ is
  constant on $\mathcal{T}_h^v$.
\end{proof}

Having proven discrete stability with inexact quadrature, the bounds on the
quadrature error let us prove optimal convergence of $u_h$ to the true
solution $u$, as well as an improved estimate for $\|u_h - P_h u\|_0$ via a
duality argument, which implies second-order pointwise convergence at the
centroids of the elements.

\begin{thm}
  Let $(\sigma, u, p)$ be the solution of \eqref{eq:hr}, and let
  $(\sigma_h, u_h, p_h)$ be the solution of
  \[
    \begin{aligned}
      a_h(\sigma_h,\tau_h) + c_h(\sigma_h,q_h) + c_h(\tau_h,p_h) +
      b(\sigma_h,v_h) + b(\tau_h,u_h) &= (g,v_h), & &
      (\tau_h,v_h,q_h) \in
      \Sigma_h \times V_h \times \tilde Q_h.
    \end{aligned}
  \]
  Then assuming $\sigma$, $u$, and $p$ are smooth enough, there exists $C_1$
  independent of $h$ such that
  \begin{align}
      \|\sigma - \sigma_h\|_0 + \|p - p_h\|_0 &\leq C_1 h \|A\|_{1,\infty}
      (\|\sigma\|_1 + \|p\|_1),\label{eq:stressconv} \\
      |\sigma - \sigma_h|_{\div} &\leq C_1 h \|\div
      \sigma\|_1,\label{eq:divconv}\\
      \|u - u_h\|_0 &\leq C_1 h
      \|A\|_{1,\infty}(\|\sigma\|_1 + \|p\|_1 + \|u\|_1).\label{eq:uconv}
  \end{align}
  If \eqref{eq:hr} is sufficiently regular that there exists $C_2$ independent
  of $g$ such that
  \begin{equation}\label{eq:u2}
    \|\sigma\|_1 + \|p\|_1 + \|u\|_1 \leq C_2 \|g\|_0,
  \end{equation}
  then there exists $C_3$ independent of $h$ such that
  \[
    \|u_h - P_h u\|_0 \leq C_3 h^2 \|A\|_{2,\infty}(\|\div \sigma\|_1 +
    \|\sigma\|_1 + \|p\|_1).
  \]
\end{thm}

\begin{proof}
  (Cf. \textcite[Theorems~3.4, 4.1, 4.3]{WheelerIotov2006}.)

  Let $\xi := (\sigma,p)$, $\xi_h := (\sigma_h,p_h)$, $K_h :=
  \ker(b,\Sigma_h;V_h) \times \tilde Q_h = \ker(\div,\Sigma_h) \times \tilde
  Q_h$, and let $f$, $f_h$, $E$ and $\Pi_h$ be as in \cref{prop:quaderror}.

  \Cref{prop:lbbchquad} proves that there exists $\beta > 0$ independent of $h$
  such that
  \[
    \inf_{\xi_h \in K_h} \sup_{\tau_h \in K_h} f_h(\xi_h,\tau_h) \geq \beta
    \|\xi\|_0 \|\tau\|_0.
  \]
  Therefore, because $\div \Pi_h^{\div} \sigma = P_h g = \div \sigma_h$,
  we have $\Pi_h \xi - \xi_h\in K_h$, and so by \eqref{eq:quaderrorh}
  \[
    \begin{aligned}
      \|\Pi_h \xi - \xi_h \|_0
      &\leq
      \frac{1}{\beta}
      \sup_{\zeta_h \in K_h}
      \frac{f_h(\Pi_h \xi - \xi_h, \zeta_h)}
      {\|\zeta_h\|_0}
      \\
      &=
      \frac{1}{\beta}
      \sup_{\zeta_h \in K_h}
      \frac{f_h(\Pi_h \xi,\zeta_h)) + b(\zeta_{h,\Sigma},u_h)}
      {\|\zeta_h\|_0}
      \\
      &=
      \frac{1}{\beta}
      \sup_{\zeta_h \in K_h}
      \frac{f(\xi,\zeta_h) - E(A,\xi,\zeta_h)}
      {\|\zeta_h\|_0}
      \\
      &=
      \frac{1}{\beta}
      \sup_{\zeta_h \in K_h}
      \frac{b(\zeta_{h,\Sigma},u) - E(A,\xi,\zeta_h)}
      {\|\zeta_h\|_0}
      \\
      &\leq C h \|A\|_{1,\infty}\|\xi\|_1.
    \end{aligned}
  \]
  The bound \eqref{eq:stressconv} follows by the triangle inequality and the
  approximation properties of $\Pi_h^{\div}$.

  The bound \eqref{eq:divconv} is a standard result of the commuting property
  of $\Pi_h^{\div}$ and is unchanged by inexact quadrature.

  Due to the stability
  of $b$ on $\Sigma_h \times V_h$ and \eqref{eq:quaderrorh},
  \[
    \begin{aligned}
      \|u_h - P_h u\|_0 &\leq C
      \sup_{\tau_h \in \Sigma_h}
      \frac{b(\tau_h,u_h - P_h u)}{\|\tau_h\|_{\div}} \\
      &=
      C
      \sup_{\tau_h \in \Sigma_h}
      \frac{f(\xi,(\tau_h,0)) - f_h(\xi_h,(\tau_h,0))}{\|\tau_h\|_0}
      \\
      &=
      C
      \sup_{\tau_h \in \Sigma_h}
      \frac{f(\xi - \Pi_h \xi,(\tau_h,0)) - f_h(\xi_h - \Pi_h \xi,(\tau_h,0)) +
      E(A,\Pi_h \xi,(\tau_h,0))}{\|\tau_h\|_0} \\
      &\leq
      Ch\|A\|_{1,\infty}\|\xi\|_1.
    \end{aligned}
  \]
  The bound \eqref{eq:uconv} follows by the triangle inequality and the
  approximations properties of $P_h$.

  Now let us assume that \eqref{eq:u2} holds, and let $(\tau,v,q)$ be the
  solution of \eqref{eq:hr} with $g = P_h u - u_h$.  Because $g = P_h g$ in
  this case,
  \[
    \|P_h u - u_h\|_0^2
    =
    b(\tau, P_h u - u_h)
    = 
    b(\Pi_h \tau, P_h u - u_h).
  \]
  Defining $\zeta := (\tau, q)$, we have
  \[
    \begin{aligned}
      b(\Pi_h \tau, P_h u - u_h) =\
      &f_h(\xi_h,\Pi_h \zeta) - f(\xi,\Pi_h \zeta)\\
      =\ 
      &f(\xi_h - \xi,\Pi_h \zeta) - E(A,\xi_h,\Pi_h \zeta)
      \\
      =\ 
      &f(\xi_h - \xi,\zeta) + f(\xi_h - \xi, \Pi_h \zeta - \zeta)- E(A,\xi_h,\Pi_h \zeta)
      \\
      =\ 
      &b(\sigma - \sigma_h,v) + f(\xi_h - \xi, \Pi_h \zeta - \zeta)- E(A,\xi_h,\Pi_h \zeta)
      \\
      =\ 
      &b(\sigma - \sigma_h,v) + f(\xi_h - \xi, \Pi_h \zeta - \zeta)
      - E(A,\Pi_h \xi,\Pi_h \zeta)
      - E(A,\xi_h - \Pi_h \xi, \Pi_h \zeta).
    \end{aligned}
  \]
  Because $\div(\sigma) - \div(\sigma_h) = (I - P_h) \div \sigma$, it is
  orthogonal to $V_h$ and
  \[
    |b(\sigma - \sigma_h,v)| = |b(\sigma - \sigma_h,v - P_h v)|
    \leq
    C h^2 \|\div \sigma\|_1 \|v\|_1.
  \]
  By this result, the already established bound for $\|\xi_h - \xi\|_0$, the
  approximation properties of $\Pi_h = (\Pi_h^{\div}, P_h)$, and the
  quadrature bounds \eqref{eq:quaderrorh} and \eqref{eq:quaderrorh2}, we have
  \[
    \begin{aligned}
      \|P_h u - u_h\|_0^2
      &\leq
      C h^2 (\|\div \sigma\|_1 \|v\|_1 + (\|A\|_{1,\infty} + \|A\|_{2,\infty})(\|\sigma\|_1 +
      \|p\|_1)(\|\tau\|_1 + \|q\|_1))\\
      &\leq
      C h^2 \|A\|_{2,\infty}(\|\div \sigma\|_1 + \|\sigma\|_1 +
      \|p\|_1)(\|v\|_1 + \|\tau\|_1 + \|q\|_1).
    \end{aligned}
  \]
  Invoking \eqref{eq:u2} completes the proof.
\end{proof}

\section{Discussion}

Although we have proved optimal and improved convergence estimates of the
$\Sigma_h \times V_h \times \tilde Q_h$ element for arbitrary order, the
existence of high order discretizations of linear elasticity whose stress is
pointwise symmetric (\cite{AdamsCockburn2004,ArnoldAwanouWinther2008,Hu2017})
means that our element may only have practical advantages for low orders.  The
limit of strongly-symmetric finite elements composed of ``pure'' polynomial
elements (i.e.\ without bubble functions) appears to be $r=3$ \cite{Hu2017}.
It is worthwhile to compare our element to other low order offerings.

\begin{itemize}
  \item
    \textcite{ArnoldFalkWinther2007} introduced $\Sigma_h \times V_h \times
    Q_h$, against which our element has already been compared.  We can
    estimate the sizes of $|Q_h|$ and $|\tilde Q_h|$ using the Euler
    characteristic of $\Omega$ and the heuristics $|\Delta_2| \approx
    2|\Delta_3|$ and $|\Delta_1| \approx 7|\Delta_0|$ \cite{Stenberg2010},
    giving us an approximation $|\Delta_1| \approx 7/6 |\Delta_3|$.  The
    lowest order $Q_h$ space has three degrees of freedom per cell, while
    $\tilde Q_h$ has two degrees of freedom per edge, which predicts $|\tilde
    Q_h|\approx 7/9|Q_h|$.  The same paper, however, introduced a reduced
    element with a stress space $|\Sigma_h^-| = 2/3 |\Sigma_h|$.

    Neither $\Sigma_h\times Q_h$ nor $\Sigma_h^- \times Q_h$ admits a jointly
    nodal discretization, so reduction to just the displacement variables is
    not a local operation.  $\Sigma_h$ alone, however, can be locally
    eliminated with corner quadrature, leaving a finite volume discretization
    of $V \times Q$ (in keeping with the interpretation of the elasticity
    complex as a resolution of the rigid-body motions).
  \item
    \textcite{HuZhang2016} developed a pointwise conservative and symmetric
    discretization $\hat{\Sigma}_{1,h}^+ \times V_{1,h}$, where $V_{1,h}$ is
    the discontinuous approximation space of rigid-body motions on each
    tetrahedron, and $\hat{\Sigma_{1,h}^+}$ is the union of
    $\mathcal{P}_1(\mathcal{T}_h;\mathbb{S})\cap C^0$ with $\div$-conforming
    bubble functions on each facet.  This space has the advantage of being
    strongly-symmetric, and its size is $6(|\Delta_0| +
    |\Delta_2| + |\Delta_3|)\approx 18.9|\Delta_0|$, compared to the lowest
    order element introduced here, $|\Sigma_h \times V_h \times \tilde Q_h| =
    2|\Delta_1| + 9|\Delta_2| + 3|\Delta_3| \approx 23.3|\Delta_3|$.  The
    stress space $\hat{\Sigma}_{1,h}^+$ would seem to admit a nodal
    discretization and quadrature, meaning that their element could also be
    reduced to a first-order finite volume scheme, though with six degrees of
    freedom per cell instead of three.

    Suppose we wished to solve a problem with a nonlinear stress-strain
    relationship, such that the stress degrees of freedom could not be
    eliminated a~priori.  The block-diagonal matrix for the approximation
    $c_h$ would still allow one to locally project $\tilde Q_h$ out of
    $\Sigma_h$ a~priori, leaving a system of the size $|\Sigma_h| + |V_h| -
    |\tilde Q_h|\approx 18.7|\Delta_3|$, making our element more competitive.
    Our element also has the practical advantage of being composed of function
    spaces that are already widely implemented.
  \item
    \textcite{AmbartsumyamKhattatovYotov2017} recently proposed a finite
    element for linear elasticity $\Sigma_h \times V_h \times \check Q_h$ for
    $r=0$, with the difference from the present work being that $\check Q_h :=
    \mathcal{P}_{r+1}(\mathcal{T}_h;\mathbb{K}) \cap C^0(\mathbb{K})$, i.e.\
    the multiplier field is fully continuous.  Their approach also reduces to
    a finite volume method, and is a mixed finite element presentation of the
    finite volume method of \textcite{Nordbotten2014}.  The similarity of our
    element to theirs is not intentional, but not surprising, as both
    approaches are generalization of the MFMFE method of
    \textcite{WheelerIotov2006}.  Though unpublished, the proof of stability
    they describe in \cite{AmbartsumyamKhattatovYotov2017} suggests that it is
    the same as \cref{prop:lbbchquad}.  As both elements appear to have the
    same approximation properties, the notable difference is the size of the
    multiplier spaces ($\check Q_h \subset \tilde Q_h$): $|\check Q_h| = 3
    |\Delta_0| \approx 1/2 |\Delta_3|$ vs. $|\tilde Q_h| = 2 |\Delta_1|
    \approx 7/3|\Delta_0|$.  Their element thus requires less work to compute
    the displacement Schur complement, while ours results in a more nearly
    symmetric stress, and a smaller approximation space when only the
    multipliers are projected out.
  \item
    \textcite{GopalakrishnanGuzman2011} and \textcite{ArnoldAwanouWinther2014}
    have introduced strongly-symmetric but $H(\div)$-nonconforming
    discretizations.  The reduced element in
    \textcite{ArnoldAwanouWinther2014} appears to be the smallest: its size is
    $9|\Delta_2| + 6|\Delta_3| \approx 24|\Delta_3|$, and its convergence is
    first-order in the $L^2$ norm of displacement and stress in general, and
    the convergence in displacement is second-order with full elliptic
    regularity.  These $H(\div)$-nonconforming methods require careful
    selection of the stress degrees of freedom to ensure stability, so a lack
    of nodal discretization would appear to rule out local elimination of the
    stress degrees of freedom.  However, because all degrees of freedom are in
    cells and on facets, the methods are hybridizable, reducing to symmetric
    positive definite systems in Lagrange multipliers on the facets.  The
    lowest order cases require three multipliers on each facet, meaning their
    reduced size is $\approx 6 |\Delta_3|$.
\end{itemize}

\printbibliography

\end{document}